\documentclass[11pt]{amsart}

\usepackage{mathrsfs}
\usepackage{amscd}
\usepackage{amsmath}
\usepackage{amssymb}
\usepackage{amsthm}
\usepackage{epsf}
\usepackage{verbatim}
\usepackage{enumitem}
\usepackage{mathtools}
\usepackage{xcolor}
\usepackage{array}
\usepackage{bm}
\usepackage{bbm}
\usepackage{graphicx}
\input epsf.tex
\usepackage{amssymb} 
\usepackage{tikz}
\usetikzlibrary{decorations.pathreplacing}
\usetikzlibrary{arrows}

\def\acts{\curvearrowright}

\usepackage{mathrsfs,xfrac} 
\usepackage[colorlinks=true,linkcolor=blue,citecolor=blue,urlcolor=blue]{hyperref} 
\usepackage[left=2.5cm, right=2.5cm, top=2cm]{geometry}
\usepackage{todonotes}

\newtheorem{theorem}{Theorem}
\newtheorem{lemma}[theorem]{Lemma}

\theoremstyle{definition}
\newtheorem{remark}[theorem]{Remark}

\theoremstyle{definition}

\theoremstyle{definition}

\newtheorem{oproblem}[theorem]{Problem}
\theoremstyle{definition}
\newtheorem{definition}[theorem]{Definition}

\newcommand{\Z}{\mathbb{Z}}

\newcommand{\ve}{\varepsilon}

\newcommand{\abs}[1]{\left\vert#1\right\vert}

\title[Ineffectiveness of the regularity lemma for bounded degree graphs]{The ineffectiveness of the regularity lemma for bounded degree graphs}
\author[Lyons]{Clark Lyons}
\address{
Clark Lyons \\
HUN-REN Alfr\'ed R\'enyi Institute of Mathematics,
Re\'altanoda u.~13-15, Budapest 1053, Hungary.
}
\email{lyons.clark@renyi.hu}

\author[Terlov]{Grigory Terlov}
\address{
Grigory Terlov \\
University of North Carolina at Chapel Hill,
Department of Statistics and Operations Research,
Chapel Hill, NC,
USA,
and
HUN-REN Alfr\'ed R\'enyi Institute of Mathematics,
Re\'altanoda u. 13-15, Budapest 1053, Hungary.
}
\email{gterlov@unc.edu}
\author[Vidny\'anszky]{Zolt\'an Vidny\'anszky}
\address{
Zolt\'an Vidny\'anszky \\
E\"{o}tv\"{o}s Lor\'{a}nd University, Institute of Mathematics, P\'azm\'any P\'{e}ter stny. 1/C, 1117 Budapest, Hungary
}
\email{zoltan.vidnyanszky@ttk.elte.hu}

\begin{document}
\begin{abstract}
We show that for any $\Delta \geq 3$, there is no bound computable from $(\varepsilon, r)$ on the size of a graph required to approximate a graph of maximum degree at most $\Delta$ up to $\varepsilon$ error in $r$-neighborhood statistics. This provides a negative answer to a question posed by Lovász. Our result is a direct consequence of the recent celebrated work of Bowen, Chapman, Lubotzky, and Vidick, which refutes the Aldous--Lyons conjecture.
\end{abstract}
\maketitle

\section{Introduction and main results}
Szemerédi’s regularity lemma and its variants provide a foundational framework for understanding dense graphs. Roughly speaking, the lemma asserts that every graph can be approximated, in terms of homomorphism densities, by a bounded-size graph, where the size bound is explicit and depends on the chosen error parameter.

To be more precise, for finite graphs $F$ and $G$ define the \textbf{homomorphism density of $F$ in $G$} by
\[
t(F,G)=\frac{\left|\mathrm{Hom}(F,G)\right|}{|V(G)|^{|V(F)|}},
\]
that is, the probability that a uniformly random function from the vertices of $F$ to the vertices of $G$ defines a graph homomorphism. The following is a well-known consequence of the Frieze--Kannan (weak) regularity lemma with bounds given in \cite{FriezeKannan}. 

\begin{theorem}[{\cite{FriezeKannan}}]
	\label{t:szemeredi}
There is an absolute constant $C$ such that for any $K$, any $\ve>0$, and any graph $G$, there is a graph $G'$ with $|V(G')|\le C^{K^2/\ve^2}$ such that $$|t(F,G)-t(F,G')|<\ve$$ for all graphs $F$ with $|E(F)|\le K$.
\end{theorem}

We note that there are variants of this theorem for notions of sampling based on injective homomorphisms and graph embeddings. The theorem can also be adapted to the setting of graphs equipped with edge and vertex labelings (see \cite[Chapter 9]{lovasz2012large} and references therein).

For sparse graphs, any homomorphism density is close to zero and hence such an approximation is not useful. Instead, one approximates them by neighborhood statistics, defined as follows. Suppose that $G$ is a finite graph with degrees bounded by $\Delta$. A \textbf{rooted graph of radius $\le r$} is a graph $F_\bullet$ with a distinguished root vertex such that all vertices of $F_\bullet$ are of distance $\le r$ from the root. Rooted graphs are considered up to root-preserving isomorphism $\cong_\bullet$ of graphs. For a rooted graph $F_\bullet$ of radius $\le r$ and a (finite) graph $G$ define the corresponding $r$-\textbf{neighborhood statistic} in $G$ as
\begin{equation}\label{eq:neigbstat}
    u_{r}(F_\bullet,G)=\frac{\big|\{v\in V(G)\mid B_r(v)\cong_\bullet F_\bullet\}\big|}{|V(G)|},
\end{equation}
which is equal to the probability that the radius $r$ ball around a uniformly chosen vertex of $G$ is isomorphic, as a rooted graph, to $F_\bullet$.

Lov\'asz has asked for an analogue of Theorem \ref{t:szemeredi} in the case of bounded degree graphs, and the following observation was soon made by Alon \cite{unpub}.

\begin{theorem}[Alon \cite{unpub}, as in {\cite[Theorem 4.8.4]{Zhao_2023}}]\label{thm:Alon_regularity}
For a fixed $\Delta>0$ there is a function $N_\Delta(\ve,r)$ such that for every graph $G$ with maximum degree at most $\Delta$, there is a graph $G'$ of maximum degree at most $\Delta$ with $|V(G')|\le N_\Delta(\ve,r)$ such that 
\[
    \abs{u_{r}(F_\bullet,G)-u_{r}(F_\bullet,G')}<\ve
\] 
for all rooted graphs $F_\bullet$ of radius $\le r$. 
\end{theorem}

This statement follows from the total-boundedness of the space of possible $r$-neighborhood statistics for graphs (see Section \ref{sec:proof_of_main} for the proof). Thus, unlike the case of dense graphs, the function $N_\Delta(\varepsilon,r)$ is not constructive. It was then asked by Lov\'asz \cite[page 358]{lovasz2012large} (see also \cite[Open Problem 4.8.5]{Zhao_2023}) whether it is possible to find an ``effective'' bound for $N_\Delta(\ve,r)$. He also suggested that an efficient, algorithmic way of constructing an approximating graph could be used also on graphings, and hence to find a finitary approximation, in turn, resolving positively the famous Aldous--Lyons conjecture \cite{AldousLyons}. 

In a recent breakthrough, Bowen, Chapman, Lubotzky, and
Vidick \cite{bowen2024aldouslyonsconjectureisubgroup,bowen2024aldouslyonsconjectureiiundecidability} gave a negative solution to the latter conjecture. In this note we show that their main undecidability result directly implies a negative answer to Lov\'asz's question:

\begin{theorem}[Regularity is ineffective for bounded degree graphs]\label{thm:ineffectiveness}
For any $\Delta\geq 3$, any regularity bound $N_\Delta(\ve,r)$ is not computable.
\end{theorem}

Bowen, Chapman, Lubotzky, and
Vidick employed so-called subgroup tests to encode non-local games and then built on the ideas from the earlier deep results of Ji, Natarajan, Vidick, Wright, and Yuen \cite{ji2022mipre} to deduce undecidability. 
We note that although the subgroup test approach might not be necessary for their original goal \cite{manzoor2025equivalencerelationvonneumann} it is indeed essential for our argument. Let us also mention that solely a counterexample to the Aldous--Lyons conjecture in itself would have been likely insufficient to show Theorem \ref{thm:ineffectiveness}, but, the insight provided by the undecidability results suffices for this purpose. 

Furthermore, there seems to be an unnoticed implication in the converse direction. Namely, we provide an easy argument showing that a positive answer to the Aldous--Lyons conjecture would have implied a computable bound on $N_\Delta(\ve,r)$. This observation may still be of interest, as one could attempt to directly show that such a bound does not exist, thereby providing an alternative proof of the failure of the conjecture.

\subsection*{Organization}
In Section~\ref{sec:proof_of_main} we state the undecidability result (as a combination of the statements from \cite{bowen2024aldouslyonsconjectureisubgroup,bowen2024aldouslyonsconjectureiiundecidability}) and we prove the analogous result to Theorem~\ref{thm:ineffectiveness} stated for Schreier graphs (Theorem~\ref{main}), which is equivalent to Theorem~\ref{thm:ineffectiveness} by Lemma~\ref{lem:equiv}.
Section~\ref{sec:graphings} is dedicated to the upper approximation of $r$-neighborhood statistics of graphings, and showing the implication mentioned above.  In Section \ref{sec:decision} we describe a concrete decision function, the inexistence of which would also imply the failure of the Aldous-Lyons conjecture. 
In Section~\ref{sec:open} we state several open problems.
Finally, in Appendix~\ref{ssec:encoding} we present standard encodings between graphs and Schreier graphs and prove Lemma~\ref{lem:equiv}.

\section{Proof of Theorem~\ref{thm:ineffectiveness}}\label{sec:proof_of_main}

Both sampling regularity for dense graphs and neighborhood statistics regularity are well known to extend to decorated graphs. We will work with graphs equipped with a Schreier decoration of an action of a free group on finitely many generators.

\begin{definition}
For a natural number $d$, an $\mathbb{F}_d$ \textbf{Schreier graph} $G$ consists of a vertex set $V(G)$, a symmetric (but not necessarily reflexive) edge set $E(G)\subseteq V(G)^2$, and an edge labeling $c_G:E(G)\to\mathcal{P}(\{a_1,\ldots,a_d,a_1^{-1},\ldots,a_d^{-1}\})$ such that:
\begin{itemize}
    \item we have $a_i\in c_G(x,y)\iff a_i^{-1}\in c_G(y,x)$
    \item for any $x\in V(G)$ and any $\ell\in\{a_1,\ldots,a_d,a_1^{-1},\ldots,a_d^{-1}\}$ there is a unique $y\in V(G)$ such that $(x,y)\in E(G)$ and $\ell\in c_G(x,y)$. 
\end{itemize}
The edges and edge labels specify the same information as an action of the free group on $n$ generators $a_G:\mathbb{F}_d\curvearrowright V(G)$.
\end{definition}

For any $\mathbb{F}_d$ Schreier graph $G$, around any vertex $v\in V(G)$ we can consider a ball $B_r(G)$ of radius $r$ as a rooted edge-labeled (non-simple) graph. We consider these graphs up to root-preserving and label-preserving graph isomorphism $\cong^*_\bullet$. This allows us to formulate neighborhood statistics for $\mathbb{F}_d$ Schreier graphs exactly as before but with this new notion of isomorphism. Just as in the undecorated case, there is a regularity bound obtained by a compactness (or, more precisely, total-boundedness) argument, see Section \ref{sec:graphings}.

\begin{definition}
For a rooted edge-labeled graph $F_\bullet$ of radius $\le r$ and an $\mathbb{F}_d$ Schreier graph $G$, define the \textbf{Schreier neighborhood statistic} 
\[
    u^*_r(F_\bullet,G)=\frac{\big|\{v\in V(G)\mid B_r(v)\cong^*_\bullet F_\bullet\}\big|}{|V(G)|}.
\]
An $\mathbb{F}_d$ \textbf{Schreier regularity bound} is a function $N^*_d(\ve,r)$ such that for any $\mathbb{F}_d$ Schreier graph $G$ there is an $\mathbb{F}_d$ Schreier graph $G'$ with $|V(G')|\le N^*(\ve,r)$ such that 
\[
    \abs{u^*_r(F_\bullet,G)-u^*_r(F_\bullet,G')}<\ve
\]
for all rooted edge-labeled graphs $F_\bullet$ of radius $\le r$.
\end{definition}

We are now ready state a version of Theorem~\ref{thm:ineffectiveness} for Schreier graphs. 

\begin{theorem}[Regularity is ineffective for Schreier graphs]\label{main}
For $d\geq2$ any Schreier regularity bound $N_d^*(\ve,r)$ is not a computable function.
\end{theorem}

This version is equivalent to the original theorem for graphs by the following lemma. The proof of this equivalence relies on a well known fact that finite graphs can be encoded by Schreier graphs and vice versa. We present this equivalence in Appendix~\ref{ssec:encoding} in full detail.

\begin{lemma}\label{lem:equiv}
Let $\Delta\geq3$ and $d\geq2$. A sparse regularity bound $N_\Delta(\varepsilon,r)$ can be computed from any Schreier regularity bound $N_d^*(\varepsilon,r)$ as an oracle and vice versa.
\end{lemma}

The powerset $\mathcal{P}(\mathbb{F}_d)$ carries compact Hausdorff topology from viewing it as a product of discrete spaces $2^{\mathbb{F}_d}$. Let $\mathscr{T}_d$ denote the set of all continuous rational-valued functions $\mathcal{T}$ on $\mathcal{P}(\mathbb{F}_d)$. By continuity, for any such function $\mathcal{T}$ there is a finite set $K\subseteq\mathbb{F}_d$ such that for any $S\subseteq\mathbb{F}_d$, then output $\mathcal{T}(S)$ is determined by $S\cap K$.


The \textbf{value} of a continuous rational function $\mathcal{T}$ with respect to an $\mathbb{F}_d$ Schreier graph is defined by 
\[
    \text{val}(\mathcal{T},G)=\mathbb{E}_{v\in V(G)}\left[\mathcal{T}(\text{Stab}(v))\right],
\] 
where $\text{Stab}(v)$ is the stabilizer of $v$ in the induced action of $\mathbb{F}_d$ on $V(G)$ and the expectation is over a uniform random choice of vertex $v\in V(G)$. The \textbf{sofic value} of a continuous rational function is the supremum of its values with respect to all (finite) $\mathbb{F}_d$ Schreier graphs 
\[
    \text{val}_{sof}(\mathcal{T})=\sup_{G}\text{val}(\mathcal{T},G).
\]

Then for any $n$ the set of rational lower bounds for sofic values 
\[
    \{(\mathcal{T},q)\in\mathscr{T}_d\times\mathbb{Q}\mid\text{val}_{sof}(\mathcal{T})>q\}
\]
is a computably enumerable set, through a brute-force search of all $\mathbb{F}_d$ Schreier graphs. 

The following theorem summarizes a key step of the argument of Bowen, Chapman, Lubotzky, and
Vidick in their resolution of the Aldous--Lyons conjecture. See \cite[Theorems 1.13 and 7.3]{bowen2024aldouslyonsconjectureisubgroup} and \cite[Theorem 1.1]{bowen2024aldouslyonsconjectureiiundecidability} for the general result.

\begin{theorem}[\cite{bowen2024aldouslyonsconjectureisubgroup,bowen2024aldouslyonsconjectureiiundecidability}]\label{undec}
For any $d\geq2$ there is a computable map which takes as input a Turing machine $M$ and outputs a continuous rational function $\mathcal{T}^M\in\mathscr{T}_d$ and rational number $\lambda^M>0$ such that
\begin{itemize}
    \item if $M$ halts then $\text{val}_{sof}(\mathcal{T}^M)=1$,
    \item if $M$ does not halt then $\text{val}_{sof}(\mathcal{T}^M)\le 1-\lambda^M$.
\end{itemize}
\end{theorem}


We briefly note how the result of Bowen, Chapman, Lubotzky, and
Vidick is stronger than what is quoted above. For them, the continuous rational function $\mathcal{T}^M$ is restricted to be a subgroup test, which is a certain kind of continuous rational function with outputs in $[0,1]$ which mimics an interactive proof system between ``prover'' and ``verifier'' entities. Additionally, in the case that $\text{val}_{sof}(\mathcal{T}^M)=1$, there is a (finite) $\mathbb{F}_d$ Schreier graph $G$ which achieves this supremum so $\text{val}_{sof}(\mathcal{T}^M,G)=1$. The proof of Theorem~\ref{main} only uses that the sofic values of continuous rational functions cannot be computably approximated from above uniformly.


\begin{proof}[Proof of Theorem~\ref{main}]
We show how to computably enumerate rational upper bounds for sofic values of a continuous rational function from a Schreier regularity bound $N^*_d(\ve,r)$.
Consider a continuous rational function $\mathcal{T}$ whose output is determined by intersection with the finite subset $K\subseteq\mathbb{F}_d$. Let $r$ be a radius large enough that in the Cayley graph of $\mathbb{F}_d$, the ball of radius $r$ around the identity contains $K$. 
Note that for a vertex $v$ in an $\mathbb{F}_d$ Schreier graph $G$, the value $\mathcal{T}(\text{Stab}(v))$ depends only on the edge-labeled rooted subgraph $F_\bullet$ induced by $B_r(v)$, considered up to a root and label preserving isomorphism. 
We denote this value $\mathcal{T}(\text{Stab}(F_\bullet))$.
Let $\mathcal{F}^*_d(r)$ be the finite set that consists of all rooted graphs of radius $\le r$ with edges labeled by the standard generating set of $\mathbb{F}_d$, considered up to a rooted and label-preserving isomorphism. 
Let $m(\mathcal{T})=\max_{F_\bullet\in \mathcal{F}^*_d(r)}\mathcal{T}(\text{Stab}(F_\bullet))$.
For any $\Theta>0$ if 
\[
\abs{u^*_r(F_\bullet,G)-u^*_r(F_\bullet,G')}<\ve=\frac{\Theta}{|\mathcal{F}^*_d(r)|\cdot m(\mathcal{T})}
\]
for all $F_\bullet\in\mathcal{F}^*_d(r)$, then 
\[
    \left|\text{val}(\mathcal{T},G)-\text{val}(\mathcal{T},G')\right|\le\sum_{F_\bullet\in\mathcal{F}^*_d(r)}|u^*_r(F_\bullet,G)-u^*_r(F_\bullet,G')|\cdot\mathcal{T}(\text{Stab}(F_\bullet))<\Theta.
\] 

Thus for each $\Theta>0$ we can compute all $\mathbb{F}_d$ Schreier graphs with at most $N^*_d(\ve,r)$ vertices for this value of $\ve$, and we can compute the maximum value 
\[
    \beta_{\mathcal{T},\Theta}=\max_{|V(G')|\le N^*_d(\ve,r)}\text{val}(\mathcal{T},G')
\]
among these graphs. Then by the Schreier regularity bound we have 
\begin{equation}
	\label{e:approx}
    \beta_{\mathcal{T},\Theta}\le\text{val}_{sof}(\mathcal{T})\le \beta_{\mathcal{T},\Theta}+\Theta.
\end{equation}

But this contradicts Theorem \ref{undec}: indeed, given a Turing-machine $M$, we could use Theorem \ref{undec} to calculate $\mathcal{T}^M$ and $\lambda^M$. Then, using $\mathcal{T}^M$ we could find an $r$ as above and calculate $\beta_{\mathcal{T},\Theta}$ for $\Theta=\lambda^M/2$. Finally, \eqref{e:approx} would guarantee that $M$ halts iff $\beta_{\mathcal{T},\Theta}>1-\lambda^M$. 

\end{proof}

\section{Local statistics of graphings}\label{sec:graphings}

Let $\mathcal{F}_\Delta(r)$ denote the (finite) set of rooted graphs of maximal degree $\le \Delta$ and radius $\le r$ around the root, considered up to rooted graph isomorphism. Similarly, let $\mathcal{G}_\Delta$ denote the (infinite) set of finite graphs of maximum degree $\le \Delta$, also considered up to an isomorphism. There is a natural map $U_{r}:\mathcal{G}_\Delta\to[0,1]^{\mathcal{F}_\Delta(r)}$ that encodes $r$-neighborhood statistics of $G\in \mathcal{G}_\Delta$ in terms of all possible neighborhood statistics $\{u_{r}(F_\bullet,G)\}_{F_\bullet\in\mathcal{F}_\Delta(r)}$, namely
\[
U_{r}(G):=\{F_\bullet\mapsto u_{r}(F_\bullet,G)\}.
\]
Since $[0,1]^{\mathcal{F}_\Delta(r)}$ is compact in the metric $d_\infty$, the image $U_{r}(\mathcal{G}_\Delta)$ is totally bounded. 

So for any $\ve>0$ there is a finite subset $\mathcal{N}\subseteq\mathcal{G}_\Delta$ such that for all $G\in \mathcal{G}_\Delta$ there exists $G'\in\mathcal{N}$ such that $d_\infty(U_{r}(G),U_{r}(G'))<\ve$ which implies 
\[
\abs{u_{r}(F_\bullet,G)-u_{r}(F_\bullet,G')}<\ve
\] 
for all rooted graphs $F_\bullet$ of radius $\le r$. 

\begin{remark}
In fact, taking $N_\Delta(\ve,r)$ to be the maximum number of vertices of a graph in $\mathcal{N}$ yields a proof of Theorem~\ref{thm:Alon_regularity} as presented in \cite[Theorem 4.8.4]{Zhao_2023} (with the only difference that it is stated in terms of $\ell^\infty$ distance instead of the total variation distance, but since $\mathcal{F}_\Delta$ is finite, these distances are equivalent).
\end{remark}

The same argument proves the version of Theorem~\ref{thm:Alon_regularity} for $\mathbb{F}_d$ Schreier graphs and the map $U^*_{d,r}:\mathcal{G}^*_d\to[0,1]^{\mathcal{F}_d^*(r)}$ from finite $\mathbb{F}_d$ Schreier graphs to $r$-neighborhood statistics.

This map extends to the set of Schreier graphings obtained by measure preserving actions of $\mathbb{F}_d$ on a standard probability space. It is easy to see that the value only depends on the invariant random subgroup obtained as the stabilizer of a random vertex for such an action.

Consider the space of probability measures on the power set of $\mathbb{F}_d$, denoted as $\mathrm{Prob}(\mathcal{P}(\mathbb{F}_d))$. 
An \textbf{invariant random subgroup} of $\mathbb{F}_d$ is a Borel probability measure $\mu\in\mathrm{Prob}(\mathcal{P}(\mathbb{F}_d))$ supported on subgroups of $\mathbb{F}_d$ which is invariant under conjugation. 
We denote the set of invariant random subgroups of $\mathbb{F}_d$ by $\text{IRS}_d$ and define the corresponding $r$-neighborhood statistic as $U^*_{d,r}:\text{IRS}_d\to[0,1]^{\mathcal{F}^*_d(r)}$ as 
\[
U^*_{d,r}(\theta):=\{F_\bullet\mapsto u^*_{d,r}(F_\bullet,\theta)\},
\]
where $u^*_{d,r}(F_\bullet,H)$ is the probability that the random subgroup $\theta$ encodes $F_\bullet$, in the sense that the radius $r$ ball around a vertex with stabilizer $\theta$ is isomorphic to $F_\bullet$ as a rooted edge-labeled graph.
Notice that the Aldous--Lyons conjecture is equivalent to the statement that $U_{d,r}^*(\mathcal{G}_{d})$ is dense in $U_{d,r}^*(\mathrm{IRS}_d)$ for any $d$ and $r$.

Given $k>0$ let $W_d(k)$ be the ball of radius $k$ around the identity in the Cayley graph of $\mathbb{F}_d$ with respect to the standard generating set. 
A $k$-\textbf{pseudo-subgroup} of $\mathbb{F}_d$ is a subset of $W_d(k)$ which is the intersection of $W_d(k)$ with a subgroup of $\mathbb{F}_d$. 
Then a $k$-\textbf{pseudo-IRS} is defined as a Borel probability measure on $\mathcal{P}(W_d(k))\subseteq \mathcal{P}(\mathbb{F}_d)$ which is invariant under conjugation by the standard generators (when this conjugation is defined). 
Similarly to the above, we denote the set of $k$-pseudo-IRSs of $\mathbb{F}_d$ by $\text{P-IRS}_d(k)$. In fact, it follows that 
\[
    \text{IRS}_d=\bigcap_k\text{P-IRS}_d(k)
\] 
as a decreasing intersection.

It is not hard to see that Theorem \ref{thm:ineffectiveness} implies that one cannot algorithmically confirm whether a given open set covers all statistics of IRSs coming from an action on a finite space (i.e., there is no Turing machine which halts precisely when the containment holds, see also \cite{bowen2024aldouslyonsconjectureisubgroup}). By contrast, using pseudo-subgroups, it can be algorithmically confirmed whether such an open set covers all the IRSs. 


\begin{theorem}\label{thm:upper_approx}
There is a Turing machine $M(d,r,S)$, which takes as input a number of generators $d$, radius $r$, and $S\subseteq[0,1]^{\mathcal{F}^*_d(r)}$ a union of finitely many open $d_\infty$-neighborhoods (centered at rational points) such that $M(d,r,S)$ halts if and only if $U^*_{d,r}(\text{IRS}_d)\subseteq S$.
\end{theorem}
\begin{proof}
First, we note that the value $U^*_{d,r}(\theta)$ is determined by its restriction to $W_d(2r)$ (which is a pseudo-IRS). In particular, we can extend the map to $U^*_{d,r}:\text{P-IRS}_d(k)\to[0,1]^{\mathcal{F}^*_d(r)}$ for any $k\geq 2r$. 
We have 
\[
    U^*_{d,r}(\text{IRS}_d)=\bigcap_{k\geq 2r}U^*_{d,r}(\text{P-IRS}_d(k)).
\] 
Since the map $U^*_{d,r}$ is continuous and $U^*_{d,r}(\text{IRS}_d)$ is a closed set, by compactness of $U^*_{d,r}(\text{IRS}_d)$, if $U^*_{d,r}(\text{IRS}_d)\subseteq S$ then there is some $k\geq 2r$ such that 
$
    U^*_{d,r}(\text{P-IRS}_d(k))\subseteq S.
$

Since the linear inequalities defining $U^*_{d,r}(\text{P-IRS}_d(k))$ are computable uniformly in $d$ and $k$ \cite[Lemma 2.16]{bowen2024aldouslyonsconjectureisubgroup}, and since checking the consistency of finite families of rational inequalities is computable, the condition $U^*_{d,r}(\text{P-IRS}_d(k))\subseteq S$ is also computable. 
Thus, the Turing machine $M(d,r,S)$ can search through all $k\geq 2r$ and halt when it determines $U^*_{d,r}(\text{P-IRS}_d(k))\subseteq S$.
\end{proof}

\begin{remark}\label{rem:upper_approx}
It follows that if for all $d$ and $r$ the $U_{d,r}^*(\mathcal{G}_{d})$ were dense in $U_{d,r}^*(\mathrm{IRS}_d)$ (i.e.\ the Aldous--Lyons conjecture holds) then, by Theorem~\ref{thm:upper_approx}, there would have been a computable upper bound to $N^*_d(\ve,r)$. 
Indeed, a brute force search through finite $\mathbb{F}_d$ Schreier graphs would eventually enumerate a finite collection whose $r$-neighborhood statistics form an $\ve$-cover of $U_{d,r}^*(\mathcal{G}_{d})$, equivalently of $U^*_{d,r}(\text{IRS}_d)$, and will detect this in finite time.
\end{remark}

\section{A Decision Problem for Local Statistics}
\label{sec:decision}

In light of Remark \ref{rem:upper_approx}, there is a potential alternative approach to resolving the Aldous-Lyons conjecture avoiding non-local games by directly showing that for some $\Delta\geq3$ there is no computable sparse regularity bound $N_\Delta(\epsilon, r)$. We point out that this problem can be rephrased as a decision problem about graphs with specified local statistics.

\begin{definition}
    Fix $\Delta\geq3$. A \textbf{local statistic decision function} (LSDF) is a function $\Phi_\Delta(\varepsilon, r, S)$ that takes as input a tolerance $\varepsilon>0$, a radius $r$, and a set $S\subseteq [0,1]^{\mathcal{F}_\Delta(r)}$ (specified by a finite list of inequalities) and outputs ``yes'' or ``no'' such that:
    \begin{itemize}
        \item if the output is ``yes'' then there is a finite graph of degree at most $\Delta$ whose $r$-neighborhood statistic is in the set $S$,
        \item if the output is ``no'' then there is no finite graph of degree at most $\Delta$ such that the $\varepsilon$ ball around the $r$-neighborhood statistic is a subset of $S$.
    \end{itemize}
\end{definition}

Note that if we allowed $\varepsilon=0$ in the LSDF, then the function would just determine whether there exists a finite graph with specified local statistics. Such a function cannot be computable since the local statistics of a bounded degree graph can encode the execution of a Turing machine, such that the machine halts if and only if a finite graph with these local statistics exists. The stronger statement that there is no computable LSDF is in fact equivalent to the non-computability of a sparse regularity bound $N_\Delta(\varepsilon,r)$.

\begin{theorem}
\label{t:LSDF}
Fix $\Delta\geq3$. A local statistic decision function $\Phi_\Delta(\varepsilon,r,S)$ can be computed from a sparse regularity bound $N_\Delta(\varepsilon,r)$ and vice versa.
\end{theorem}

\begin{proof}
First assume we have oracle access to a sparse regularity bound $N_\Delta(\varepsilon,r)$. We describe an algorithm to compute a local statistic decision function $\Phi_\Delta(\varepsilon,r,S)$.

Given a tolerance $\varepsilon$, a radius $r$, and a set of local statistics $S$, we first enumerate all graphs of degree at most $\Delta$ with at most $N_\Delta(\varepsilon,r)$ vertices. If the local statistics of one of these graphs is in $S$, then output ``yes'' and otherwise output ``no''.

If the output is ``yes'' then clearly there is a finite graph with $r$-local statistics in $S$. Now suppose there is a finite graph $G$ such that the $\varepsilon$ ball around the $r$-neighborhood statistics of $G$ is contained in $S$. By regularity, there is a graph $G'$ with at most $N_\Delta(\varepsilon,r)$ many vertices whose $r$-neighborhood statistics are within $\varepsilon$ of $G$. Thus the $r$-neighborhood statistics of $G'$ is also in $S$ so $\Phi_\Delta(\varepsilon,r,S)$ outputs ``yes'' as desired.

Conversely, assume that we have oracle access to a local statistic decision function $\Phi_\Delta(\varepsilon,r,S)$. Given $\varepsilon$ and $r$ we recursively build a finite set of graphs $\mathscr{G}$ such that every graph of degree at most $\Delta$ has $r$-neighborhood statistics within $\varepsilon$ of some graph in $\mathscr{G}$. Call such a set $\mathscr{G}$ an $\varepsilon$-net. Then we let $N_\Delta(\varepsilon,r)$ be the maximum number of vertices of a graph in the $\varepsilon$-net $\mathscr{G}$.

Initially we let $\mathscr{G}$ be empty and we recursively apply the following. Let $S$ be the complement in $[0,1]^{\mathcal{F}_\Delta(r)}$ of the union of the $\varepsilon/2$ balls around the statistics of each element of $\mathscr{G}$. If $\Phi_\Delta(\varepsilon/2,r,S)$ returns ``no'' then we conclude that $S$ does not contain the $\varepsilon/2$ ball around the statistics of any graph of degree at most $\Delta$. This implies that any graph of degree at most $\Delta$ is within $\varepsilon$ of some graph in $\mathscr{G}$ in $r$-neighborhood statistics, so $\mathscr{G}$ is an $\varepsilon$-net. If $\Phi_\Delta(\varepsilon/2,r,S)$ returns ``yes'' then there is some $G$ of degree at most $\Delta$ with $r$-neighborhood statistics in $S$. By going through a fixed enumeration of finite graphs, we can find such a $G$ with minimal cardinality. We add $G$ to $\mathscr{G}$ and iterate.

It remains to argue that this recursive procedure eventually returns a ``no'' response so that the process stops. So suppose that an infinite sequence is enumerated into $\mathscr{G}=\{G_1,G_2,\ldots\}$. First, we argue that the infinite set $\mathscr{G}$ is an infinite $\varepsilon/2$-net. Otherwise let $G$ be a graph of degree at most $\Delta$ which is not within $\varepsilon/2$ of any $G_i$ in $r$-neighborhood statistics. Let $k$ be such that $G_k$ has more vertices than $G$. Then by minimality we have that $G$ is enumerated into $\mathscr{G}$ as $G_k$ because it has $r$-neighborhood statistics at least $\varepsilon/2$ away from each of $G_1,\ldots,G_{k-1}$. This is a contradiction, and so $\mathscr{G}$ is an infinite $\varepsilon/2$ net. By compactness there is a finite subset $\{G_1,\ldots,G_j\}$ which is also a $\varepsilon/2$-net. But then the procedure would have stopped at stage $j$, because there are no graphs of degree at most $\Delta$ with $r$-neighborhood statistics at least $\varepsilon/2$ from each of $G_1,\ldots,G_j$.

\end{proof}

\section{Open Problems}\label{sec:open}

We now present several open questions. First, Theorem \ref{t:LSDF}, i.e., the rephrasing as a decision problem motivates a potential alternative approach to the resolution of the Aldous-Lyons conjecture:
\begin{oproblem}
Is it possible to directly encode Turing machines into local statistics of bounded degree graphs to show that there is no computable LSDF?    
\end{oproblem}

The present note is concerned with the bounded degree graphs, and, as mentioned in the introduction, the analogous statement is not true in the other end of the spectrum, i.e., for dense graphs. It is possible to build a continuous spectrum connecting bounded degree graphs to dense graphs, parametrized by a real, which measures their density, and to ask appropriate versions of approximation theorems (see, e.g., \cite{backhausz2022action,frenkel2018convergence} for different variants). This naturally leads to the following question.

\begin{oproblem} Can one describe the phase transition from approximable to inapproximable, as the density decreases?
\end{oproblem}

It would be extremely interesting for example, if, as the density decreases, the bound in Theorem \ref{t:szemeredi} would increase, eventually yielding a non-computable growth rate.

Another natural direction could be the following. Let $\mathcal{A}_r\subseteq \mathbb{R}^{\mathcal{F}_\Delta(r)}$ be the set of all probability distributions of $r$-neighborhoods around the root in $G$, where $G$ ranges over all finite uniformly-rooted graphs. Similarly, let $\mathcal{A}_r'\subseteq \mathbb{R}^{\mathcal{F}_\Delta(r)}$ be the same, but where $G$ ranges through all graphings. Clearly, the closure of $\mathcal{A}_r$, denoted by $\overline{\mathcal{A}_r}$, is a contained in $\mathcal{A}_r'$. The Aldous--Lyons conjecture can be stated as $\overline{\mathcal{A}_r}=\mathcal{A}_r'$ for every $r$. The following problems were mentioned by Lov\'asz \cite[page 358]{lovasz2012large} as potential ways of refuting the Aldous--Lyons conjecture. In light of Theorem~\ref{thm:ineffectiveness} we still find them interesting.
\begin{oproblem}
    Find an explicit $r>0$ such that $\overline{\mathcal{A}_r}\neq \mathcal{A}_r'$.
\end{oproblem}

One can also restrict the attention to some finite set of rooted graphs $F_1,F_2,\ldots,F_m \in {\mathcal{F}_\Delta(r)}$ and to each finite graph $G$ one can associate a vector $(u_r(F_1,G),\ldots,u_r(F_m,G))\subset \mathbb{R}^m$. Let $U_r(F_1,F_2,\ldots,F_m)$ denote the set of all such vectors as $G$ ranges over all finite uniformly-rooted graphs. Again, let $U'_r(F_1,F_2,\ldots,F_m)$ be the same, but where $G$ ranges through all graphings.

\begin{oproblem}
    Given $r>0$ and finite graphs $\{F_i\}_{i=1}^m$, determine the sets ${\overline{U}_r(F_1,F_2,\ldots,F_m)}$ and ${U'_r(F_1,F_2,\ldots,F_m)}$.
\end{oproblem}

Harangi \cite{Harangi} considered this question with $m=2$, and $F_1,F_2$ are cycles of length $3$ and $4$. For any given $r$ his result gives an explicit description of the corresponding sets and yields that they do coincide. We also would like to highlight another question of Ab\'ert stated in \cite[Question 1.3]{Harangi} about the approximation of the expected spectral measures of graphings by the eigenvalue distribution of finite graphs.

\appendix

\section{Encoding graphs and Schreier graphs}\label{ssec:encoding}

The goal of this appendix is to prove the equivalence of Theorem~\ref{thm:ineffectiveness} and Theorem~\ref{main} by encoding free group Schreier graphs with bounded degree graphs and vice versa. It is a well known fact that these objects encode each other (see e.g.~\cite{Gross,AldousLyons,grigorchuk2011some,Cannizzo,TothURG} and references therein) and is often mentioned without further explanation. Although the particular encoding presented below is also known to the specialists we were not able to find where it is written explicitly.

First, there is a straightforward way of encoding any locally finite graph as an $\mathbb{F}_2$ Schreier graph. For a (simple, connected) graph $G=(V,E)$, we build an $\mathbb{F}_2$ Schreier graph $G^*$ with vertex set equal to the set of directed edges $V(G^*)=E(G).$ For each $v\in V(G)$ fix a cyclic ordering of the edges with source $v$. This defines an action on $a:\Z\acts V(G^*)$. There is also an action $b:\Z\acts V(G^*)$ which exchanges the source and target of the directed edge. Together these define an $\mathbb{F}_2$ action.

\tikzset{
		Big dot/.style={
			circle, inner sep=0pt, 
			minimum size=2.5mm, fill=black
		}
	}
\begin{figure}[htb]
\begin{center}
	\begin{tikzpicture}[thick, scale=0.7]
	\node[Big dot] (L1) at (-2,-1) {};
    \node[Big dot] (L2) at (-4,0) {};
    \node[Big dot] (L3) at (-2,1) {};
	\node[Big dot] (L4) at (-6,0) {};
    \node[Big dot] (L5) at (-6,2) {};
    \node[Big dot] (L6) at (-6,-2) {};
    \node[Big dot] (L7) at (-8,0) {};
    \draw [line width=0.5mm] (L3)-- (L1) -- (L2) -- (L4) -- (L7);
    \draw [line width=0.5mm] (L3)-- (L2);
    \draw [line width=0.5mm] (L5)-- (L6);
    
    \draw [dashed,->] (-1.5,0) -- (0,0);

    \node[Big dot] (R7) at (2,0) {};
    \node[Big dot] (R41) at (4,0) {};
    \node[Big dot] (R42) at (5,1) {};
	\node[Big dot] (R43) at (6,0) {};
    \node[Big dot] (R44) at (5,-1) {};
    \node[Big dot] (R5) at (5,3) {};
    \node[Big dot] (R6) at (5,-3) {};
    \node[Big dot] (R21) at (8,0) {};
    \node[Big dot] (R22) at (9,1) {};
    \node[Big dot] (R23) at (9,-1) {};
    \node[Big dot] (R31) at (11,2) {};
    \node[Big dot] (R32) at (12.2,2) {};
    \node[Big dot] (R11) at (11,-2) {};
    \node[Big dot] (R12) at (12.2,-2) {};

    \draw [->] (R7) to [bend left=30] node [above, sloped] (TextNode1) {$b$} (R41);
    \draw [<-] (R7) to [bend right=30] node [below, sloped] (TextNode1) {$b$} (R41);
    \draw [->] (R41) to node [above,sloped] (TextNode1) {$a$} (R42);
    \draw [->] (R42) to node [above, sloped] (TextNode1) {$a$} (R43);
    \draw [->] (R43) to node [below, sloped] (TextNode1) {$a$} (R44);
    \draw [->] (R44) to node [below, sloped] (TextNode1) {$a$} (R41);
    \draw [->] (R42) to [bend left=30] node [left] (TextNode1) {$b$} (R5);
    \draw [<-] (R42) to [bend right=30] node [right] (TextNode1) {$b$} (R5);
    \draw [->] (R6) to [bend left=30] node [left] (TextNode1) {$b$} (R44);
    \draw [<-] (R6) to [bend right=30] node [right] (TextNode1) {$b$} (R44);
    
    \draw [->] (R43) to [bend left=30] node [above] (TextNode1) {$b$} (R21);
    \draw [<-] (R43) to [bend right=30] node [below] (TextNode1) {$b$} (R21);

    \draw [->] (R21) to node [above,sloped] (TextNode1) {$a$} (R22);
    \draw [->] (R22) to node [right] (TextNode1) {$a$} (R23);
    \draw [->] (R23) to node [below, sloped] (TextNode1) {$a$} (R21);

    \draw [->] (R22) to [bend left=30] node [above] (TextNode1) {$b$} (R31);
    \draw [<-] (R22) to [bend right=30] node [below] (TextNode1) {$b$} (R31);

    \draw [->] (R31) to [bend left=20] node [above] (TextNode1) {$a$} (R32);
    \draw [<-] (R31) to [bend right=20] node [below] (TextNode1) {$a$} (R32);
    
    \draw [->] (R23) to [bend left=30] node [above] (TextNode1) {$b$} (R11);
    \draw [<-] (R23) to [bend right=30] node [below] (TextNode1) {$b$} (R11);

    \draw [->] (R11) to [bend left=20] node [above] (TextNode1) {$a$} (R12);
    \draw [<-] (R11) to [bend right=20] node [below] (TextNode1) {$a$} (R12);
    
    \draw [->] (R32) to [bend left=20] node [right] (TextNode1) {$b$} (R12);
    \draw [<-] (R32) to [bend right=20] node [left] (TextNode1) {$b$} (R12);

    \draw[->] (5,3) .. controls (4,4.5) and (6,4.5) .. node[midway, above] {$a$}(5.2,3.15);
    \draw[->] (5,-3) .. controls (4,-4.5) and (6,-4.5) .. node[midway, below] {$a$}(5.2,-3.15);
    \draw[->] (2,0) .. controls (0.5,-1) and (0.5,1) .. node[midway, left] {$a$}(1.85,0.2);
    \end{tikzpicture}	
\end{center}
\caption{Example of encoding a graph as an $\mathbb{F}_2$ Schreier graph.}
\end{figure}
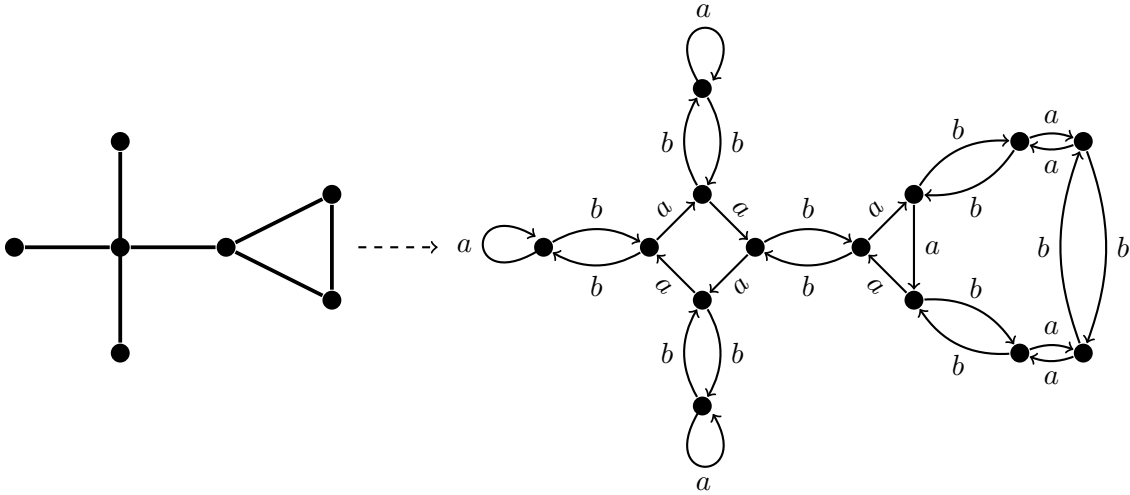

For the converse direction let $G$ be an $\mathbb{F}_d$ Schreier graph. We encode this into a simple graph $\hat{G}$ of maximum degree $3$. For each vertex $v$ of $G$ and each $\ell\in\{a_1,\ldots,a_d,a_1^{-1},\ldots,a_d^{-1}\}$ we have a vertex $v_\ell\in \hat{G}$ and these vertices are arranged in a cycle. Additionally, for each $i\in\{1,\ldots,d\}$ and edge $(x,y)$ labeled by $a_i$, there are vertices and edges forming a gadget that connects $x_{a_i}$ and $y_{a_i^{-1}}$ as in Figure~\ref{fig:direct_edges}. Therefore the only cycles of length $2d$ in $\hat{G}$ are are those corresponding to vertices of $G$, and each vertex of such a cycle is the start or end of a distinct gadget, one corresponding to each $\ell\in\{a_1,\ldots,a_d,a_1^{-1},\ldots,a_d^{-1}\}$.

\tikzset{
		big dot/.style={
			circle, inner sep=0pt, 
			minimum size=1mm, fill=black
		}
	}
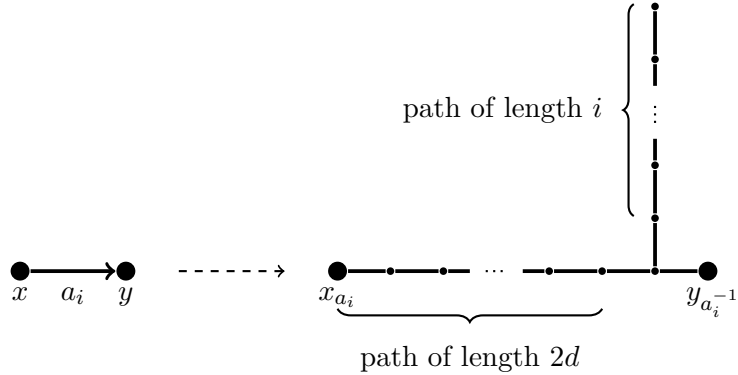
\begin{figure}[htb]
\begin{center}
    \begin{tikzpicture}[thick, scale=0.7]
    \node[Big dot] (C1) at (-2,0) {};
    \node[Big dot] (C2) at (-4,0) {};
    \draw (-3,-0.1) node[below] {$a_i$};
    \draw (-4,-0.1) node[below] {$x$};
    \draw (-2,-0.1) node[below] {$y$};
    \draw [line width=0.5mm,->] (C2) -- (C1);

    \draw [dashed,->] (-1,0) -- (1,0);
    \node[Big dot] (C3) at (2,0) {};
    \draw (2,-0.1) node[below] {$x_{a_i}$};
    \node[Big dot] (C4) at (9,0) {};
    \draw (9.1,-0.1) node[below] {$y_{a_i^{-1}}$};
    \node[big dot] (d1) at (3,0) {};
    \node[big dot] (d2) at (4,0) {};
    \node[big dot] (d3) at (6,0) {};
    \node[big dot] (d4) at (7,0) {};
    \node[big dot] (d5) at (8,0) {};
    \draw [line width=0.5mm] (C3) -- (d1) -- (d2) -- (4.5,0);
    \draw [dotted] (4.8,0) to (5.2,0);
    \draw [line width=0.5mm] (5.5,0) -- (d3) -- (d4) -- (d5) -- (C4);
    \draw [decorate,decoration={brace,amplitude=5pt,mirror,raise=3ex}] (2,0) -- (7,0) node[midway,yshift=-3em]{path of length $2d$};
    \node[big dot] (b1) at (8,1) {};
    \node[big dot] (b2) at (8,2) {};
    \node[big dot] (b3) at (8,4) {};
    \node[big dot] (b4) at (8,5) {};
    \draw [line width=0.5mm] (d5) -- (b1) -- (b2) -- (8,2.5);
    \draw [dotted] (8,2.8) to (8,3.2);
    \draw [line width=0.5mm] (8,3.5) -- (b3) -- (b4);
    \draw [decorate,decoration={brace,amplitude=4pt},xshift=-0.4cm,yshift=1.2pt] (8,1) -- (8,5) node [midway,right,xshift=-3.2cm] {path of length $i$}; 
    \end{tikzpicture}	
\end{center}
\caption{Example of the encoding of a directed edge}\label{fig:direct_edges}
\end{figure}

Both of these encodings are defined uniquely up to the choice of cyclic ordering at each vertex. Additionally, both encodings are injective maps and therefore do not lose information. It will be important in the following proof that both of these encodings are defined locally, and that they can both be decoded with small errors in the space of neighborhood statistics.

\begin{proof}[Proof of Lemma~\ref{lem:equiv}]

First fix $\Delta\geq3$ and suppose we have oracle access to a Schreier regularity bound $N_2^*(\varepsilon,r)$. We show how to compute a sparse regularity bound $N_\Delta(\varepsilon,r)$.

First let $G$ be graph of maximum degree at most $\Delta$, and let $G^*$ be its encoding as a Schreier graph as defined above. We know that for any $\varepsilon_0$ and $r_0$ there is an $\mathbb{F}_2$ Schreier graph $G^*\text{$'$}$ with $|V(G^*\text{$'$})|\leq N^*_2(\varepsilon_0,r_0)$ such that $$|u_{r_0}(F_\bullet,G^*)-u_{r_0}(F_\bullet,G^*\text{$'$})|<\varepsilon_0$$ for any rooted edge-labeled graph $F_\bullet\in\mathcal{F}^*_2(r_0)$. The graph $G^*\text{$'$}$ may not be in the image of the encoding map, but we can still decode some simple graph of degree at most $\Delta$ from the incidence of $a$-cycles and $b$-cycles in the action of $\mathbb{F}_2$ on $V(G^*\text{$'$})$.

We can say that a cycle for the action of $a$ on $V(G^*\text{$'$})$ is `good' if it is of length at most $\Delta$, and if each vertex of the $a$-cycle is part of a $b$-cycle of length $2$ with a distinct vertex outside of the $a$-cycle. We then form a graph $G'$ of degree at most $\Delta$ with `good' cycles as vertices and edges between `good' cycles if there is a $b$-cycle of length $2$ between them. By choosing $\varepsilon_0$ small enough and $r_0$ large enough we can ensure that at most $\varepsilon$ fraction of the vertices of $G^*\text{$'$}$ are not part of good cycles, since every vertex of $G^*$ is part of a good cycle and this is a locally property (i.e.~determined by a neighborhood with uniformly bounded radius).

In fact, by choosing $\varepsilon_0$ small and $r_0$ large enough, we can ensure $$|u_r(F_\bullet,G)-u_r(F_\bullet,G')|<\varepsilon$$ for any rooted graph $F_\bullet\in \mathcal{F}_{\Delta} (r)$. Therefore we can let $N_\Delta(\varepsilon,r)$ be the maximum number of vertices in one of these decoded graphs $G'$.

Conversely fix $d\geq2$ and suppose we have oracle access to a sparse regularity bound $N_3(\varepsilon,r)$. We show how to compute a Schreier regularity bound $N^*_d(\varepsilon,r)$.

Let $G$ be an $\mathbb{F}_d$ Schreier graph, and let $\hat{G}$ be its encoding as a simple rooted graph of degree at most $3$. For any $\varepsilon_0$ and $r_0$ there is a graph $\hat{G}'$ of degree at most $3$ with $|V(\hat{G}')|\leq N_3(\varepsilon_0,r_0)$ such that $$|u_{r_0}(F_\bullet,\hat{G})-u_{r_0}(F_\bullet,\hat{G}')|<\varepsilon_0$$ for any rooted graph $F_\bullet$ of radius $r_0$. As before, the graph $\hat{G}'$ may not be in the image of the encoding map, but we can decode some $\mathbb{F}_d$ Schreier graph from the small cycles and the gadgets connecting them.

We say that a cycle in $\hat{G}'$ is `good' if it has length $2d$ and each vertex is the start or end of a distinct gadget, one corresponding to each $\ell\in\{a_1,\ldots,a_d,a_1^{-1},\ldots,a_d^{-1}\}$, as in the encoding construction. We can form an $\mathbb{F}_d$ Schreier graph $G'$ with `good' cycles as vertices and edges between `good' cycles if there is a gadget between them. The symbol $a_i$ (or $a_i^{-1}$) will be part of the label of an edge if that edge (or its reverse) came from the corresponding gadget. So far, this only defines a partial $\mathbb{F}_d$ Schreier graph. For each vertex $v$ and $\ell\in\{a_1,\ldots,a_d,a_1^{-1},\ldots,a_d^{-1}\}$ there is at most one edge out of $v$ with $\ell$ as part of its label. But for each $i\in\{1,\ldots,n\}$ we have an equal number of vertices missing $a_i$ as part of an edge coming out as those missing $a_i$ as part of an edge coming in. This can be seen by analyzing the components of the subgraph of edges labeled $a_i$ or $a_i^{-1}$, which has degree at most $2$ and therefore consists of cycles and paths. Therefore we can add an additional edge from each vertex with $a_i$ as a missing out-label to a vertex with $a_i$ as a missing in-label. This new edge will have $a_i$ as part of its label (and $a_i^{-1}$ in the other direction). This ensures $G'$ is an $\mathbb{F}_d$ Schreier graph. By choosing $\varepsilon_0$ small enough and $r_0$ large enough we can ensure that at most $\varepsilon$ fraction of the vertices of $\hat{G}'$ are not part of good cycles or the gadgets connected to them, since this is true of all vertices in $G^*$ and this property is determined locally.
Finally, as before, by choosing $\varepsilon_0$ small and $r_0$ large enough, we can ensure $$|u_r^*(F_\bullet,G)-u_r^*(F_\bullet,G')|<\varepsilon$$ for any labeled rooted graph $F_\bullet$ of radius $\le r$. Therefore we can let $N^*_\Delta(\varepsilon,r)$ be the maximum number of vertices in one of these decoded graphs $G'$.
\end{proof}
    
\section*{Acknowledgements}
We would like to thank Michael Chapman, L\'aszl\'o Lov\'asz, and all the participants of the PSQC reading seminar at Renyi for many insightful conversations. The second author was supported by the ERC Synergy Grant No.~810115 - DYNASNET. The 1st and 3rd authors were partially supported by HAS Momentum Grant 2022-58. The 3rd author was also supported by National Research, Development and Innovation Office (NKFIH) grant no.~146922.

\bibliographystyle{alpha}
\bibliography{AL} 

\end{document}